\documentclass[a4paper,11pt,reqno]{amsart}
\usepackage{amsfonts,amsmath,amssymb,amsthm}
\usepackage{graphicx}
\usepackage{fixmath}
\usepackage{hyperref}
\usepackage[english]{babel}

\usepackage{setspace}
\usepackage{subfig}
\usepackage{url}
\usepackage{booktabs}
\usepackage{ifthen}
\usepackage{wrapfig}

\usepackage{enumerate}

\usepackage{tikz}
\usetikzlibrary{calc}
\usetikzlibrary{decorations.pathreplacing}

\usepackage[T2A,T1]{fontenc}

\newtheorem{thm}{Theorem}[section]

\newtheorem{prop}[thm]{Proposition}

\newtheorem{defn}[thm]{Definition}

\newtheorem{conj}[thm]{Conjecture}

\numberwithin{equation}{section}

\begin{document}
 
\thispagestyle{plain}

\title{Lobsters with an almost perfect matching are graceful}

\author{Elliot Krop}
\address{Elliot Krop  \texttt{ElliotKrop@clayton.edu}}

\address{Department of Mathematics \\
Clayton State University \\
Morrow, GA 30260, USA}

\begin{abstract}
Let $T$ be a lobster with a matching that covers all but one vertex. We show that in this case, $T$ is graceful.
\\[\baselineskip] 
\\[\baselineskip]
\end{abstract}

\date{\today}

\maketitle

\section{Introduction} 
Let $T$ be a tree on $n$ vertices. Define a weight on an edge as the absolute difference of the labels of its incident vertices. The graceful tree conjecture (GTC), which has been attributed variously to Anton Kotzig, Gerhard Ringel, and Alexander Rosa states
\begin{conj}\cite{Rosa}
It is possible to label the vertices of $T$ uniquely from $0$ to $n-1$, so that the set of weights of $T$ is $\{1,2, \dots, n-1\}$.
\end{conj}

Every tree that accepts the conjectured labeling is known as a \emph{graceful tree}. 

To browse the many attempts to solve this problem and their partial results, see Gallian's Dynamic Survey \cite{Gallian}.

We use the definition of \emph{tree distance} that first appeared in \cite{Morgan}. Let $P$ be a longest path in $T$ and call $T$ a $k$-distant tree if all of its vertices are a distance at most $k$ from $P$. Paths ($0$-distant trees) and caterpillars ($1$-distant trees) were shown to be graceful in \cite{Rosa}. However, the conjecture is unknown for any trees with higher tree distance. For $2$-distant trees, or \emph{lobsters}, the question is a well-known conjecture of Jean-Claude Bermond.

\begin{conj}\cite{Bermond}
Every lobster is graceful.
\end{conj}

The $\Delta$ construction, used to create larger graceful trees from graceful trees was first defined by Stanton and Zarnke \cite{SZ} in 1973 and has since been modified and generalized by Koh, Rogers, and Tan \cite{KRT} and Burzio and Ferrarese \cite{BF}. It   is one of the earliest yet most robust methods to create graceful trees. 

A construction that has received far less attention is that of Hajo Broersma and Cornelis Hoede \cite{BH} from 1999, who showed an equivalence between the GTC and a more restrictive labeling on trees containing a perfect matching. Their results were unfamiliar enough that in subsequent years David Morgan published a paper proving that all lobsters with a perfect matching are graceful \cite{Morgan}, although this claim was already stated in slightly different language as Corollary 11 in \cite{BH}.

Let $T$ be a tree of order $n$ with a perfect matching $M$. A graceful labeling of $T$, which additionally satisfies the property that for any edge in $M$, the pair of vertices incident to that edge must have a label sum of $n-1$, is called a \emph{strongly graceful} labeling. For any tree $T$ with a perfect matching $M$, the tree resulting from the contraction of the edges of $M$ is called the \emph{contree} of $T$.

\begin{thm}[Broersma-Hoede]\cite{BH}\label{BHEquivalence}
Every tree is graceful if and only if every tree containing a perfect matching is strongly graceful.
\end{thm}

In this work, we apply the approach of Matt Superdock\cite{Superdock} who consolidated the constructions of \cite{BH} and \cite{BF} and extended Theorem \ref{BHEquivalence}.

\begin{defn}
A tree $T$ is \emph{0-rotatable} if for every vertex of $v$ of $T$, there exists a graceful labeling $f$ so that $f(v)=0$.
\end{defn}

\begin{thm}[Superdock]
If the tree $T$ has an almost perfect matching, and the contree of $T$ is $0$-rotatable, then $T$ is graceful.
\end{thm}

As a consequence of our analysis, we answer Conjecture $2.14$ of \cite{BK},

\begin{thm}\label{almost perfect}
Every lobster with a matching that covers all but one vertex, that is, an almost perfect matching, is graceful.
\end{thm}

For an excellent and detailed account of the current state of the GTC along with some new results, we refer the reader to Matt Superdock's senior thesis \cite{Superdock}, from which we borrow definitions and notation. 

For basic graph theoretic concepts, see The Book \cite{West}.

\section{$\Delta$ Constructions and Beyond}
\begin{defn}\cite{SZ}
Let $S$ and $T$ be two trees and let $v$ be a vertex of $T$. Replace each vertex of $S$ by a copy of $T$ by identifying each vertex of $S$ with the vertex corresponding to $v$ in the distinct copy of $T$. Denote the resulting tree by $S \Delta T$.  
\end{defn}

\begin{thm}\cite{SZ}
If $S$ and $T$ are graceful, then $S \Delta T$ is graceful.
\end{thm}

Suppose $S$ and $T$ have orders $n_S$ and $n_T$, respectively, with graceful labelings $f$ and $g$. The construction used to prove the above theorem requires $n_S$ copies of $T$, each substituted for a vertex of $S$. Let $(A,B)$ be a bipartition of $T$. For $0\leq i \leq n_S-1$, the following function labels the vertices of each copy of $T$. The index corresponds to the label of $S$ taken from the value of $f$ on the vertex of $S$ into which we substitute the copy of $T$.
  \[ g_i(x) = \left\{\begin{array}{ll}
    in_T+g(x), & \text{ if } x\in A\\
    (n_S - i -1)n_T+g(x), & \text{ if } x \in B\\
      \end{array}\right. \]

\begin{defn}\cite{SZ}
Let $S$ and $T$ be trees and let $u,v$ be vertices of $S$ and $T$, respectively. Replace each vertex of $S$, other than the exceptional vertex $u$, by a copy of $T$ by identifying each vertex of $S$ with the vertex corresponding to $v$ in the distinct copy of $T$. Denote the resulting tree by $S\Delta_{+1}T$.
\end{defn}

\begin{thm}\cite{SZ}
If $S$ of order $n_S$ and $T$ are trees with graceful labelings $f$ and $g$, where $f(u)=n_S-1$ and $g(v)=0$, then $S\Delta_{+1}T$ is graceful.
\end{thm}

This construction is almost identical to the one above.  Let $(A,B)$ be a bipartition of $T$ with $v\in A$. The labeling function follows, other than the label on $u$, which is $(n_S-1)n_T$.
  \[ g_i(x) = \left\{\begin{array}{ll}
    in_T+g(x), & \text{ if } x\in A\\
    (n_S - i -2)n_T+g(x), & \text{ if } x \in B\\
      \end{array}\right. \]

Burzio and Ferrarese \cite{BF} generalized both of these constructions. To improve the $\Delta$ construction, consider two adjacent vertices of $S$ into which we substitute copies of $T$. Notice that we may connect two such copies of $T$ by an edge between {\bf any} two vertices that correspond to the same vertex in each copy. For the $\Delta_{+1}$ construction, the above observatoin holds for any adjacent pair of vertices of $S$ other than the exceptional vertex $u$, which must still be adjacent to the vertices corresponding to the fixed vertex $v$ of $T$.

\begin{prop}\label{matchings}
Let $T$ be a tree with an almost perfect matching $M$. For every path of maximum length with end vertex $v$, $T$ accepts an almost perfect matching which covers all vertices but $v$.
\end{prop}

\begin{proof}
We induct on the order $n$ of $T$, which we notice must be odd. The statement is trivial for $n=3$. Next, suppose the theorem holds for all trees of order less than $n$. Let $T$ be a tree of order $n$ with an almost perfect matching $M$. If the vertex not covered by $M$ is an end vertex on a path of maximum length, then $M$ is the required matching, so suppose this is not the case. Notice that every leaf edge of $T$ must be in $M$. Choose a path of maximum length, $P$, and remove a leaf edge $e$ and its incident vertices to form the tree $T'=T\backslash \{e\}$. By the induction hypothesis, the theorem holds for $T'$. 

If the remaining portion of $P$ in $T'$, that is, $P\backslash\{e\}$, is not a path of maximum length in $T'$, then $T'$ contains a path $Q$ of the same length as $P$. By the induction hypothesis, we can produce an almost perfect matching $M'$ of $T'$ excluding an end vertex of $Q$. This means that $M'\cup\{e\}$ is an almost perfect matching for $T$ satisfying the necessary conditions.

If $P\backslash\{e\}$ is a path of maximum length in $T'$, then we apply the induction hypothesis to produce a almost perfect matching $M'$ of $T$ which excludes an end vertex of $P\backslash\{e\}$ which is not incident to $e$ in $P$. Again, $M'\cup\{e\}$ is the required almost perfect matching for $T$. 
\end{proof}

\begin{proof}{(of Theorem \ref{almost perfect})}
Let $T$ be a lobster with an almost perfect matching. By Proposition \ref{matchings}, we can find an almost perfect matching $M$ of $T$ that covers all but an end vertex $u$ on a path of maximum length. Contract the edges of $M$ to form the caterpillar $T'$. Rosa showed that caterpillars are graceful \cite{Rosa} and produced a graceful labeling in which any end vertex on a path of maximum length may be labeled $0$. Notice that for any graceful labeling $f$ of a tree with $m$ edges, the \emph{complementary labeling} $\bar{f}$ is defined by $\bar{f}(v)=m-f(v)$ for each vertex $v$ of the tree, is graceful. We label $T'$ by the complementary labeling of Rosa's labeling, so that $u$ receives the maximum label. Notice that $u$ serves as our exceptional vertex in a generalized $\Delta_{+1}$ construction with $T'$ as $S$ and $P_2$ as $T$.
\end{proof}

\end{document}